\documentclass{article}

\usepackage{verbatim}
\usepackage[pdftex]{graphicx}
\usepackage{amsmath, amssymb, amsfonts, amsthm}
\usepackage{setspace}
\usepackage{enumitem}
\usepackage{tabularx}
\usepackage{thmtools}
\usepackage{thm-restate}
\usepackage{sidecap}
\usepackage{tikz}
\usepackage{mathdots}
\usetikzlibrary{calc}
\usetikzlibrary{decorations.pathreplacing}
\usetikzlibrary{positioning,patterns}
\usetikzlibrary{arrows,shapes,positioning}
\usetikzlibrary{decorations.markings}
\sidecaptionvpos{figure}{c}

\def \smvx {circle[radius = .08][fill = black]}

\tikzstyle{edge}=[very thick]
\tikzstyle{diredge}=[postaction={decorate,decoration={markings,
		mark=at position .65 with {\arrow[scale = 1]{stealth};}}}]
\newcommand{\defPt}[3]{
	\def \pt {(#1, #2)}
	\coordinate [at = \pt, name = #3];
}

\newcommand{\ignore}[1]{}

\setlist{nolistsep}

\def \PFq[#1]{\mathbf{P}^#1(\mathbb{F}_q)}
\def \PF[#1,#2]{PG(#1,#2)}

\def \eps{\varepsilon}
\DeclareMathOperator{\per}{Perm}
\DeclareMathOperator{\match}{PM}
\DeclareMathOperator{\fix}{Fix}
\declaretheorem[name=Theorem]{theorem}
\newtheorem{conjecture}[theorem]{Conjecture}
\newtheorem{lemma}[theorem]{Lemma}

\newtheorem{corollary}[theorem]{Corollary}
\newtheorem{proposition}[theorem]{Proposition}

\newtheorem{observation}[theorem]{Observation}
\newtheorem{claim}[theorem]{Claim}

\begin{document}

\title{Perfect matchings and derangements on graphs.}
\author{Matija Bucic\thanks{ETH Z\"urich, Supported by SNSF grant 200021-175573}, Pat Devlin\thanks{Yale University.}, Mo Hendon\thanks{University of Georgia.}, Dru Horne\footnotemark[1], Ben Lund\thanks{Princeton University. Supported by DMS-1344994.}}
\maketitle

\begin{abstract}
We show that each perfect matching in a bipartite graph $G$ intersects at least half of the perfect matchings in $G$.
This result has equivalent formulations in terms of the permanent of the adjacency matrix of a graph, and in terms of derangements and permutations on graphs.
We give several related results and open questions.
	
	\begin{comment}
Let $G=(E,V)$ be a directed, loopless graph.
A derangement on $G$ is a bijection $\sigma:V \rightarrow V$ such that $(v,\sigma(v)) \in E$ for all $v \in V$.
A permutation on $G$ is a bijection $\sigma:V \rightarrow V$ such that either $(v, \sigma(v)) \in E$ or $v = \sigma(v)$ for all $v \in V$.

Let $(d/p)_G$ denote the ratio of the number of derangements on $G$ to the number of permutations on $G$.
It is a classical result that $\lim_{n \rightarrow \infty} (d/p)_{K_n} = e^{-1}$, and it is easy to see that $(d/p)_G$ will be $0$ when there are no derangements on $G$.
How large can $(d/p)_G$ be?

Our main result is that $(d/p)_G \leq 1/2$, with equality only for a directed cycle.
We conjecture that, if $G$ is an undirected graph on $2n$ vertices, then $(d/p)_G \leq (d/p)_{K_{n,n}}$, and we prove this for bipartite graphs.

One corollary of our main result is that, if $G$ is a bipartite graph and $M$ is a perfect matching in the complement of $G$ that respects a $2$-coloring of $G$, then $G \cup M$ has at least twice as many perfect matchings as $G$.
\end{comment}
\end{abstract}

\section{Introduction}

Our main result concerns perfect matchings in bipartite graphs.
\begin{theorem}\label{th:matchingInBipartiteGraph}
	Let $G$ be a bipartite graph having a perfect matching $M$.
	Then $M$ has non-empty intersection with at least half of the perfect matchings in $G$.
\end{theorem}

The value of $1/2$ in this conclusion can't be improved, since even cycles have two perfect matchings, which are disjoint.  Moreover, the hypothesis that $G$ be bipartite is necessary (e.g., $K_4$ has three perfect matchings, which are disjoint). In fact, it turns out that for general graphs, the behavior is quite different.
\begin{theorem}\label{th:matchingInGeneral}
Let $G$ be a graph on $2n$ vertices having a perfect matching $M$.  Then $M$ has non-empty intersection with at least $1/(2^{n-1}+1)$ of the perfect matchings in $G$.  Moreover, if $n$ is even there is a $3$-regular graph on $2n$ vertices having $2^{n/2} + 1$ perfect matchings, one of which is disjoint from all the others.
\end{theorem}

%We also show that this example is relatively close to being the worst possible.

%\begin{theorem}\label{th:matchingInGeneralBound}
%Let $G$ be a graph on $2n$ vertices and let $M$ be a perfect matching in $G$. Then $M$ has a non-empty intersection with at least a proportion of $1/2^n$ of all perfect matchings in $G$.
%\end{theorem}
%It is worth noting that our proof of Theorem \ref{th:matchingInGeneralBound} relies on Theorem \ref{th:matchingInBipartiteGraph}.

This work was originally motivated as a study of derangements and permutations on graphs, as introduced by Clark \cite{clark2013graph}, and this approach leads to several related questions.
Let $G=(E,V)$ be a directed, loopless graph.
A derangement on $G$ is a bijection $\sigma:V \rightarrow V$ such that $(v,\sigma(v)) \in E$ for all $v \in V$.
A permutation on $G$ is a bijection $\sigma:V \rightarrow V$ such that either $(v, \sigma(v)) \in E$ or $v = \sigma(v)$ for all $v \in V$.

For any directed graph $G$, denote by $(d/p)_G$ the ratio of the number of derangements on $G$ to the number of permutations on $G$.
We also consider derangements and permutations on undirected graphs, by treating them as directed graphs for which $(u,v)$ is an edge if and only if $(v,u)$ is an edge.
From these definitions, it is easy to see that $(d/p)_{K_n}$ is the probability that a uniformly random permutation on $[n]$ is a derangement.
A classic application of inclusion-exclusion shows that
$$\lim_{n \rightarrow \infty}(d/p)_{K_n} = 1/e.$$

Many graphs do not admit any derangements, and hence $(d/p)$ may be as small as $0$.
In the other direction, if $C$ is a directed cycle, then there is one derangement and two permutations on $C$, and hence $(d/p)_C = 1/2$.
Theorem \ref{th:matchingInBipartiteGraph} is equivalent to the claim that $(d/p)$ is never larger than $1/2$.

\begin{restatable}{theorem}{ratioHalf}\label{thm:ratioHalf}
If $G$ is a loopless directed graph, then
$$(d/p)_G \leq 1/2,$$
with equality if and only if $G$ is a directed cycle.
\end{restatable}

Let us see that the claim that $(d/p)_G \leq 1/2$ in Theorem \ref{thm:ratioHalf} is equivalent to Theorem \ref{th:matchingInBipartiteGraph}.
Given a directed graph $G$, we construct a bipartite graph $G'$ as follows.
For each vertex $v \in V(G)$, we have a left vertex $v_l$ and a right vertex $v_r$ in $V(G')$.
We have $\{u_l, v_r\} \in E(G')$ if and only if either $(u,v) \in E(G)$ or $u =v$.
Then, each permutation on $G$ corresponds to a perfect matching in $G'$, and a derangement on $G$ corresponds to a perfect matching in $G'$ that does not use any edge of the form $\{v_l, v_r\}$.
Hence, Theorem \ref{thm:ratioHalf} is equivalent to the claim that the matching $\{ \{v_l, v_r\} : v \in V(G)\}$ in $G'$ intersects at least half of the perfect matchings in $G'$.
We can obtain the full statement of Theorem \ref{th:matchingInBipartiteGraph} by applying a permutation to the labeling of the right vertices of $G'$.

The formulation in terms of derangements and permutations on graphs leads to several other natural questions.
In particular, we investigate the ratio $(d/p)$ for restricted families of graphs.

In the case that $G$ is regular and very dense, it turns out that $(d/p)_G$ is always close to $1/e$, as it is for complete graphs.

\begin{theorem}\label{thm:veryDense}
	Let $\mathcal{G} = \{G_2, G_3, \ldots\}$ be an infinite family of directed graphs, so that $G_n$ is $k_n$-regular on $n$ vertices.
	Suppose that 
	$k_n = n - o(n/\log(n))$.
	Then,
	$$\lim_{n \rightarrow \infty} (d/p)_{G_n} = 1/e.$$
\end{theorem}

By Theorem \ref{thm:ratioHalf}, $(d/p) = 1/2$ is attained only for directed cycles.
A natural question is whether it is possible for $(d/p)$ to be nearly $1/2$ for dense graphs.
It turns out that there are graphs having positive edge density and $(d/p)$ arbitrarily close to $1/2$.
\begin{theorem}\label{thm:blowupExample}
	For any $\eps > 0$, there is a constant $c_\eps > 0$, an infinite set $I$ of positive integers, and a family of graphs $\{G_n : n \in I\}$ such that $|V(G_n)| = n$, $|E(G_n)| \geq c_\eps n^2$, and $\lim_{n \rightarrow \infty} (d/p)_{G_n}$ exists and is strictly larger than $1/2 - \eps$.
\end{theorem}

We are also interested in how large $(d/p)_G$ can be for an undirected graph $G$.
In Section \ref{sec:constructions}, we show that $(d/p)_{K_{n,n}} > (d/p)_{K_{2n}}$.
However, in contrast to the case for directed graphs, this is essentially the only example we have found of an undirected graph $G$ with $(d/p)_G > (d/p)_{K_n}$.
We propose the following conjecture on the matter.

\begin{conjecture}\label{conj:undirected}
	Let $G$ be an undirected graph on $n$ vertices.
	If $n$ is even, then
	$$(d/p)_G \leq (d/p)_{K_{n/2,n/2}},$$
	with equality only for $K_{n/2,n/2}$.
\end{conjecture}

In the case that $n$ is odd, we don't know of any graphs $G$ on $n$ vertices with $(d/p)_G > (d/p)_{K_n}$; in particular, we don't know of any way to add a vertex to $K_{n,n}$ without substantially decreasing the ratio $d/p$.

The following theorem provides some evidence for Conjecture \ref{conj:undirected}.

\begin{restatable}{theorem}{undirectedBipartite}\label{thm:undirectedBipartite}
	Let $G$ be a bipartite graph.
	Then,
	$$(d/p)_G \leq \frac{1}{\sum_{k=0}^{n} k!^{-2}},$$
	with equality if and only if $G$ is a complete bipartite graph.
\end{restatable}

In the case of random graphs and digraphs, we show that the numbers of derangements and permutations are tightly concentrated about their means.  This can be derived either from the arguments of Frieze and Jerrum \cite{frieze1995analysis} on the permanents of random matrices or from the more general framework of Janson \cite{janson1994numbers} who proved limit distribution laws for a wide range of subgraph counts.

For a probability $q$, we use $G_{n,q}$ (resp.\ $DG_{n,q}$) to denote the Erd\H{o}s--R\'enyi random graph (resp.\ digraph) on $n$ vertices where each edge (resp.\ arc) is included independently with probability $q$.  It turns out that a great deal can be said about the ratio $(d/p)$ for these random structures (following \cite{janson1994numbers}); however, for both brevity and accessibility we'll restrict ourselves to the following.

\begin{proposition}\label{th:randomRatios}
For each $\delta > 0$, there is a sequence $\varepsilon_n \to 0$ satisfying the following.  If $q = q_n \in (\delta, 1]$, then with probability tending to $1$ as $n \to \infty$,
\begin{eqnarray*}
e^{-1/q} (1 - \varepsilon_n) \leq &(d/p)_{G_{n,q}}& \leq e^{-1/q} (1 + \varepsilon_n), \qquad \text{and}\\
e^{-1/q} (1 - \varepsilon_n) \leq &(d/p)_{DG_{n,q}}& \leq e^{-1/q} (1 + \varepsilon_n).
\end{eqnarray*}
\end{proposition}
Thus a large random (di)graph with edge density $q$ will have $(d/p) \sim e^{-1/q}$.  Setting $q=1/2$ shows that for large $n$, the vast majority of graphs and digraphs have $(d/p)$ tending to $e^{-2}$. Taking $q \to 1$, we obtain a randomized analog of Theorem \ref{thm:veryDense}.  Moreover, the behavior for random graphs provides some (mild) additional evidence for Conjecture \ref{conj:undirected}.

\subsection{Prior work}

Derangements and permutations on graphs were introduced by Clark, \cite{clark2013graph}, who considered mainly the cycle structure of graph derangements.

Penrice \cite{penrice1991derangements} (using a slightly different terminology) investigated the number of derangements on $n$-partite graphs with $t$ vertices in each part, for $t$ fixed and $n$ large.
His result is a special case of Theorem \ref{thm:veryDense}, and his proof is based on the same ideas used in the proof of Theorem \ref{thm:veryDense}.

The literature on perfect matchings in bipartite graphs is extensive. For a general background on this, see the textbook of Lov\'asz and Plummer \cite{lovasz2009matching}.

\begin{comment}
Other restricted families of permutations have been investigated, but these topics are not related to our notion of derangements and permutations on graphs in any simple way.
The question of determining the number of permutations that avoid particular patterns has been very fruitful, e.g. \cite{simion1985restricted}.
Geometric permutations are those that can be realized as a line transversal of convex sets, \cite{katchalski1985geometric}.
\end{comment}

\subsection{Organization of the paper}

The construction for Theorem \ref{thm:blowupExample} is in Section \ref{sec:constructions}.
The complete bipartite graph is a special case of this construction.
Section \ref{sec:veryDense} contains the proof of Theorem \ref{thm:veryDense}.
Section \ref{sec:undirectedBipartite} contains the proof of Theorem \ref{thm:undirectedBipartite}.
Theorem \ref{thm:ratioHalf} is proved in Section \ref{sec:arbitrary}, which also includes the statement and proof of a result on Hamilton cycles in directed graphs (Corollary \ref{thm:hamilton}). In Section \ref{sec:matchingsInGeneralGraphs} we prove Theorem \ref{th:matchingInGeneral}, and we discuss random graphs and digraphs in Section \ref{sec:randomGraphs}.
Some open problems are listed in Section \ref{sec:openProblems}.

\section{Constructions}\label{sec:constructions}

In this section, we compute $(d/p)$ for blowups of directed graphs.
These provide the examples described in Theorem \ref{thm:blowupExample}.
A special case is the complete, balanced, undirected bipartite graph that plays an important role in Conjecture \ref{conj:undirected} and Theorem \ref{thm:undirectedBipartite}.

Let $D_{k,l}$, for $k \geq 1$ and $l \geq 2$, be a graph on vertices $v_{ij}$ for $i \in [k]$ and $j \in [l]$, such that $(v_{ij}, v_{lm}) \in D_{k,l}$ if and only if $m = j+1 \mod l$.
Note that $D_{1,n}$ is the directed cycle on $n$ vertices, and $D_{n,2}$ is the undirected complete bipartite graph on $[n] \sqcup [n]$.
Figure \ref{fig:D25} shows $D_{2,5}$.

\begin{figure}[h]
\begin{center}
\includegraphics[scale=1]{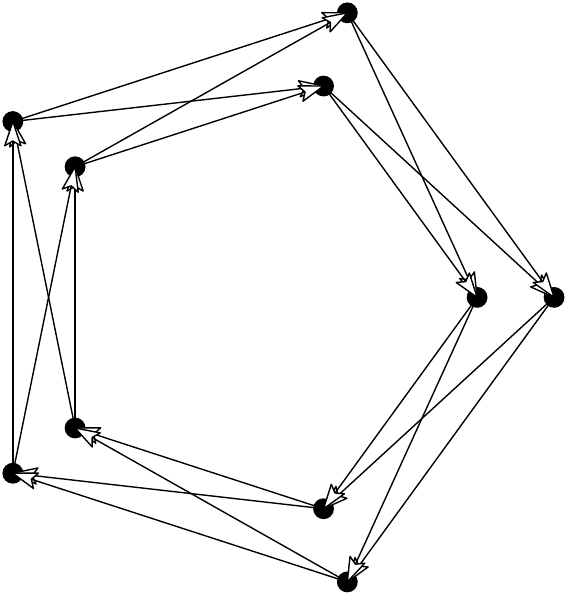}
\end{center}
\caption{$D_{2,5}$} \label{fig:D25}
\end{figure}

\begin{proposition}\label{thm:construction}
The number of derangements on $D_{k,l}$ is $(k!)^l$, and the number of permutations on $D_{k,l}$ is
$$  \sum_{i=0}^k \left( \binom{k}{i} (k-i)! \right)^l.$$
Consequently,
$$ (d/p)_{D_{k,l}} = \left( \sum_{i=0}^k \frac{1}{(i!)^l} \right)^{-1}.$$
\end{proposition}

\begin{proof}
The graph $D_{k,l}$ is organized into $l$ parts, with $k$ vertices in each part.

In any derangement on $D_{k,l}$, each vertex in part $j$ has a single out-neighbor in part $j+1 \mod l$, and no two vertices have the same out-neighbor.
The claim on the number of derangements follows immediately.

If $P$ is a permutation on $D_{k,l}$, the number of fixed points of $P$ in each part of $D_{k,l}$ must be equal.
Hence, to count the number of permutations with $m$ fixed points, it is enough to count the number of ways to choose $m/l$ fixed points in each part, and multiply by the number of derangements on the remaining elements.
This is the formula given above.
\end{proof}

Since
$e = \lim_{n \rightarrow \infty} \sum_{i=0}^n \frac{1}{i!}$
and $K_{n,n} = D_{n,2}$,
it follows from Proposition \ref{thm:construction} that $\lim_{n \rightarrow \infty} (d/p)_{K_{n,n}} > 1/e$.
For small values of $n$, it is easy to calculate that $(d/p)_{K_{n,n}} > (d/p)_{K_{2n}}$.
Hence, $(d/p)_{K_{n,n}} > (d/p)_{K_{2n}}$ for all $n>1$, consistent with Conjecture \ref{conj:undirected}.

In light of Theorem \ref{thm:ratioHalf}, it is natural to wonder if it is possible for $(d/p)_G$ to be very close to $1/2$ if $G$ is not a directed cycle.
Choosing $l$ to be a large constant shows that this is the case.
In particular, Theorem \ref{thm:blowupExample} follows from Proposition \ref{thm:construction} by taking $G_n = D_{n/l,l}$ for a suitably chosen constant $l$ depending on $\eps$.

\section{Very dense graphs}\label{sec:veryDense}

In this section, we prove Theorem \ref{thm:veryDense}.
The proof is based on two standard bounds on the permanents of matrices, the Minc-Br\'egman inequality (proved by Br\'egman in \cite{bregman1973some}) and van der Waerden conjecture (proved independently by Egorychev in \cite{egorychev1981solution} and by Falikman in \cite{falikman1981proof}).

Recall that the permanent of an $n \times n$ matrix $A = [a_{ij}]$ is
$$\per(A) = \sum_\sigma a_{1 \sigma(1)} a_{2 \sigma(2)} \ldots a_{n \sigma(n)},$$
where the sum runs over all permutations of $[n]$.
If $A$ is the adjacency matrix of a  directed graph $G$, and $\sigma$ is a fixed permutation, then
$\prod_i a_{i \sigma(i)} = 1$
if and only if $(i,\sigma(i)) \in G$ for each $i \in [n]$, and $\prod_i a_{i \sigma(i)} =0$ otherwise.
Hence, $\prod_i a_{i \sigma(i)} = 1$ if and only if $\sigma$ is a derangement on $G$, and so the permanent of $A$ is equal to the number of derangements on $G$.
Similarly, the permanent of $A+I$, where $I$ is the identity matrix, is the number of permutations on $G$.

Using the connection between derangements and permanents, Theorem \ref{thm:veryDense} follows as an immediate corollary to the following, stronger, theorem.

\begin{theorem}\label{thm:denseGraphsPermanentVersion}
Let $I$ be an infinite set of positive integers, and let $\{A_n: n \in I\}$ and $\{B_n: n \in I\}$ be families of $n \times n$ matrices with entries in $\{0,1\}$.
Suppose that the sum of entries in each row and each column of $A_n$ is $k_n$, and that the sum of entries in each row and each column of $B_n$ is $k_{n}+1$.
Further, suppose that
$k_n = n - o(n/\log(n))$.
Then,
$$\lim_{n \rightarrow \infty} \frac{Per(A_n)}{Per(B_n)} = 1/e.$$
\end{theorem}
\begin{proof}[Proof of Theorem \ref{thm:denseGraphsPermanentVersion}]
We use two standard estimates on the permanent of a matrix.
First, the Minc-Br\'egman inequality \cite{bregman1973some} is that, if $A$ is a $(0,1)$-matrix with row sums $k$, then
$$\per(A) \leq k!^{n/k},$$
with equality attained only by block-diagonal matrices.
Second, the van der Waerden Conjecture, proved independently in \cite{falikman1981proof,egorychev1981solution}, is that, if $A$ is a doubly stochastic matrix, then
$\per(A) \geq n! / n^n$,
with equality attained only by $\frac{1}{n}J$, where $J$ is the all-ones matrix.
An immediate corollary to this is that, for a matrix $A$ with row and column sums equal to $k$, we have $$\per(A) \geq n! (k / n)^n.$$
Notice the upper and lower bounds meet for $k=n$.

What remains is calculation using Stirling's approximation.
In what follows, $f(n) \sim g(n)$ means $\lim_{n \rightarrow \infty} (f(n)/g(n)) = 1$, and we denote $k=k_n$.

First, we show that $\lim_{n \rightarrow \infty} \per(A_n) / \per(B_n) \leq 1/e$.
\begin{align*}
\per(A_n)/ \per(B_n) &\leq \frac{k!^{n/k} n^n}{n!(k+1)^n}, \\
&\sim \frac{\sqrt{2 \pi k}^{n/k}}{\sqrt{2 \pi n}} \frac{ k^n}{ (k+1)^n}, \\
&\leq (2 \pi n)^{(n-k)/2k} \frac{ k^n}{ (k+1)^n}, \\
&= (2 \pi n)^{o(1/\log n)} \frac{k^n}{(k+1)^n}, \\
&\sim 1/e.
\end{align*}
In the last line, we use that $(k/(k+1))^n \sim 1/e$ for any $k = (1-o(1))n$, and $n^{1/\log n}$ is a constant so $(2 \pi n)^{o(1/\log n)} \sim 1$.

Next, we show that $\lim_{n \rightarrow \infty} \per(A_n) / \per(B_n) \geq 1/e$.
\begin{align*}
\per(A_n)/ \per(B_n) &\geq \frac{n!k^n}{(k+1)!^{n/(k+1)} n^n}, \\
&\sim \frac{\sqrt{2 \pi n}}{\sqrt{2 \pi (k+1)}^{n/(k+1)}} \frac{ k^n}{ (k+1)^n}, \\
&\geq \frac{\sqrt{2 \pi n}}{\sqrt{2 \pi n}^{n/(k+1)} - o(\sqrt{n}^{n/(k+1)})}  \frac{ k^n}{ (k+1)^n} \\
&\sim \frac{1}{(2 \pi n)^{o(1/\log n)}} (1/e), \\
&\sim 1/e.
\end{align*}
\end{proof}

\section{Undirected bipartite graphs}\label{sec:undirectedBipartite}

In this section, we prove Theorem \ref{thm:undirectedBipartite}.
As in Section \ref{sec:veryDense}, we depend on the connection between derangements on graphs and permanents of matrices.

In what follows, for any $n \times n$ matrix $M$ and sets $S,S' \subset [n]$, we denote by $M(S,S')$ the sub-matrix of $M$ consisting of rows in $S$ and columns in $S'$, and denote by $M(\overline{S}, \overline{S'})$ the sub-matrix of $M$ with rows in $S$ and columns in $S'$ deleted.
We denote by $M(\overline{i}, \overline{j})$ the sub-matrix of $M$ obtained by deleting only row $i$ and column $j$.

For any graph $G$, we denote by $p_G$ the number of permutations on $G$, and by $d_G$ the number of derangements on $G$.

First, we need the following lemma on the permanent of an arbitrary matrix.

\begin{lemma}\label{thm:subpermanentFormula}
Let $M$ be an $n \times n$ matrix, and let $0 \leq k \leq n$.
Then,
$$\binom{n}{k} \per(M) = \sum_{S, S' \in \binom{[n]}{k}} \per(M(S,S')) \per(M(\overline{S},\overline{S'})).$$
\end{lemma}

\begin{proof}
Let $S \in \binom{[n]}{k}$.
For any $S' \in \binom{[n]}{k}$, let $\Sigma_{S'}$ be the set of permutations of $[n]$ such that $\sigma(i) \in S'$ for all $i \in S$.
For each permutation $\sigma$ of $[n]$, there is a unique set $S' \in \binom{[n]}{k}$ such that $\sigma(i) \in S'$ for all $i \in S$, and hence the sets $\Sigma_{S'}$ partition the set of all permutations.

\begin{align*}
\per(M) &= \sum_{\sigma} \prod_i m_{i \sigma(i)}, \\
&= \sum_{S' \in \binom{[n]}{k}} \sum_{\sigma \in \Sigma_{S'}}\prod_{i \in S}m_{i \sigma(i)} \prod_{i \in \overline{S}} m_{i \sigma(i)}.
\end{align*}

Let $\Sigma_{S'}^1$ be the set of bijections between $S$ and $S'$, and let $\Sigma_{S'}^2$ be the set of bijections between $\overline{S}$ and $\overline{S'}$.
Each $\sigma \in \Sigma_{S'}$ corresponds to exactly one pair of functions $\sigma_1 \in \Sigma_{S'}^1$ and $\sigma_2 \in \Sigma_{S'}^2$, so that $\sigma(i) = \sigma_1(i)$ for $i \in S$ and $\sigma(j) = \sigma_2(j)$ for $j \in \overline{S}$.
Hence,

\begin{align*}
\per(M) &= \sum_{S' \in \binom{[n]}{k}}\left( \sum_{\sigma \in \Sigma_{S'}^1}\prod_{i \in S}m_{i \sigma(i)} \right) \left(\sum_{\sigma \in \Sigma_{S'}^2}\prod_{i \in \overline{S}} m_{i \sigma(i)} \right),\\
&= \sum_{S' \in \binom{[n]}{k}} \per(M(S,S')) \per(M(\overline{S}, \overline{S'})).
\end{align*}
\end{proof}

We will need to break the permutations up by their fixed points.

\begin{lemma}\label{thm:permutationSum}
If $G$ is a directed, loopless graph on $[n]$ with adjacency matrix $M$, then
$$p_G = \sum_{S \subseteq [n]} \per(M(S,S)).$$
If $G$ is an undirected bipartite graph on $[n] \sqcup [n]$ with biadjacency matrix $M$, then
$$p_G = \sum_{S, S' \subseteq [n]} \per(M(S,S'))^2.$$
\end{lemma}

\begin{proof}
Each term on the right side of the equality counts the number of permutations on $G$ with the fixed points $\overline{S}$.
In the case of a bipartite graph, the fixed points occur in pairs; $\overline{S}$ is the set of fixed points among the left vertices, and $\overline{S'}$ is the set of fixed points among the right vertices.
\end{proof}

We're now ready to prove the theorem.
For convenience, we recall it here.

\undirectedBipartite*

\begin{proof}
Let $M$ be the biadjacency matrix of $G$.
Applying Lemma \ref{thm:subpermanentFormula}, we have, for each $k$ in the range $0 \leq k \leq n$, that
\begin{align*}
d_G &= \per(M)^2,\\
&= \binom{n}{k}^{-2} \left (\sum_{S,S' \in \binom{[n]}{k}} \per(M(S,S')) \per(M(\overline{S},\overline{S'})) \right)^2.\\
\end{align*}

Note that $\per(M(S,S'))\leq k!$, with equality holding for all $S,S'$ if and only if $G=K_{n,n}$.
Hence,
$$d_G \leq k!^2 \binom{n}{k}^{-2} \left (\sum_{S,S' \in \binom{[n]}{k}} \per(M(\overline{S},\overline{S'})) \right)^2.$$
By the Cauchy-Schwarz inequality,
$$\left (\sum_{S,S' \in \binom{[n]}{k}} \per(M(\overline{S},\overline{S'})) \right)^2 \leq \sum_{S,S' \in \binom{[n]}{k}} \per(M(\overline{S},\overline{S'}))^2 \sum_{S,S' \in \binom{[n]}{k}} 1.$$
Note that this inequality is tight if $G=K_{n,n}$.
Substituting back into the expression for $d_G$, we get
\begin{equation}\label{eq:d_G} d_G \leq k!^2 \sum_{S,S' \in \binom{[n]}{k}} \per(M(\overline{S},\overline{S'}))^2.\end{equation}

On the other hand, we have by Lemma \ref{thm:permutationSum} that
\begin{equation}\label{eq:p_G}
p_G = \sum_{k=0}^n \sum_{S, S' \in \binom{[n]}{k}} \per(M(\overline{S}, \overline{S'}))^2. 
\end{equation}

Combining (\ref{eq:d_G}) and (\ref{eq:p_G}), we have
$$\frac{p_G}{d_G} \geq \sum_{k=0}^n k!^{-2},$$
with equality if and only if $G=K_{n,n}$.
\end{proof}

\section{Arbitrary directed graphs}\label{sec:arbitrary}

In this section, we prove Theorem \ref{thm:ratioHalf}.
Let us recall the statement of the theorem.

\ratioHalf*

The basic idea of the proof is to construct an injection from derangements to non-derangement permutations.

The injection that we construct is based on the cycle structure of the derangement.
As a graph, any permutation is the union of directed cycles and isolated vertices.
In particular, if $\pi(u) = w$ in a permutation $\pi$, then $uw$ is an edge in the corresponding graph.
If $\pi(u) = u$, then $u$ is an isolated vertex.
Derangements are those permutations with no isolated vertices.

Given a derangement $D$, the plan is to map $D$ to a permutation that shares all but one of the cycles of $D$.
The cycle $C$ that we break is the one that contains a specified vertex $v$.
The injection depends on the choice of $v$.
The image of $C$ will normally be a smaller cycle that contains $v$, together with at least one fixed point.
In some special cases, the image of $C$ may be a collection of fixed points, with no cycle.
Given the image of $D$, it is easy to recover all of the cycles in $D$, except for $C$.
It is easy to obtain the set of vertices in $C$, but harder to recover the order of these vertices in $C$.

Handling the cycle that contains the special vertex $v$ is the main difficulty of the proof, and we encapsulate it in the following Lemma.
\begin{lemma}\label{th:HamiltonMap}
Let $G$ be a directed, loopless graph, and let $v \in V(G)$.
Let $\mathcal{H}$ be the set of Hamilton cycles on $G$.
Let $\mathcal{G}$ be the set of permutations on $G$ with at least one fixed point and at most one cycle, such that $v$ is a vertex of the cycle if it exists.
Then, there is an injection $f_v$ from $\mathcal{H}$ to $\mathcal{G}$.
In addition, assuming that $G$ is not a directed cycle, there is a choice of $v$ for which the identity is not in the image of $f_v$.
\end{lemma}

Before proving Lemma \ref{th:HamiltonMap}, we show how it implies Theorem \ref{thm:ratioHalf}.

\begin{proof}[Proof of Theorem \ref{thm:ratioHalf}]
It is an easy observation that $(d/p)_G = 1/2$ if $G$ is a directed cycle.
Assume that $G$ is not a directed cycle.

We describe a family of injections $F_v$ from derangements on $G$ to non-derangement permutations on $G$, one injection for each vertex $v \in V(G)$.
This family of injections has the property there exists at least one $v$ for which the identity permutation is not in the image of $F_v$.
Hence, the total number of permutations is at least twice the number of derangements, plus one for the identity permutation.

For each derangement $D$ on $G$, let $F_v(D)$ be the union of all cycles of $D$ that do not contain $v$, together with the image of the cycle $C$ that contains $v$ under the function described in Lemma \ref{th:HamiltonMap} applied to the subgraph of $G$ induced by the vertices of $C$.

For the derangements that are Hamilton cycles, the map $F_v$ is exactly the map $f_v$ described in Lemma \ref{th:HamiltonMap}.
By Lemma \ref{th:HamiltonMap}, there is a choice of $v$ so that the identity permutation is not in the image of $f_v$.
For derangements that are not Hamilton cycles, each cycle that does not contain $v$ will be mapped to itself, and so the these derangements cannot map to the identity.
Hence, there is a choice of $v$ for which the identity is not in the image of $F_v$.

Given Lemma \ref{th:HamiltonMap}, it is easy to see that each permutation $P$ in the image of $F_v$ has a unique preimage $D$.
Indeed, the preimage of each cycle in $P$ that doesn't contain $v$ is the same cycle.
The preimage of the remaining vertices is their preimage under the map $f_v$ defined in Lemma \ref{th:HamiltonMap}, which is injective.
Hence, $F_v$ is injective, which is sufficient to prove Theorem \ref{thm:ratioHalf}.
\end{proof}

It remains to prove Lemma \ref{th:HamiltonMap}.

\begin{proof}[Proof of Lemma \ref{th:HamiltonMap}]
For any permutation $P$, we denote by $\fix(P)$ the set of fixed points of $P$.
The identity permutation is the unique permutation on $G$ such that $\fix(P) = V(G)$.

Let $v_0 = v$ and let $C=(v_0,v_1,v_2, \ldots,v_{n-1},v_0) \subseteq G$ be a Hamilton cycle on $G$.

Let $(v_i,v_j) \in G$ be a chord of $C$.
Note that $(v_i,v_j)$ completes a cycle $C_{ij}$ with edges from $C$, and there is a nonempty set $L_{ij}$ of vertices in $G$ that are not contained in $C_{ij}$.
Further, knowing $C$, it is easy to identify $i,j$ from either $C_{ij}$ or $L_{ij}$.

If $v_0$ is a vertex of $C_{ij}$, then we say that $(v_i,v_j)$ is a {\em forward chord} of $C$.

If there are no forward chords with respect to $C$, then $C$ is the only Hamilton cycle of $G$.
Indeed, suppose that $(v_0=u_0,u_1,u_2,\ldots,u_{n-1},u_0) \neq C$ is a Hamilton cycle of $G$.
Suppose that $(v_0, v_1, \ldots, v_i) = (u_0, u_1, \ldots, u_i)$, and $v_{i+1} \neq u_{i+1}$.
Then, $(u_i, u_{i+1})=(v_i, v_k)$ is a forward chord of $C$.
Indeed, since $k > i+1$, the sequence $(v_0, v_1, \ldots, v_i, u_{i+1} = v_k, v_{k+1}, \ldots, v_{n-1}, v_0=u_0)$ is a cycle that contains $v_0$.

We define a partial order on forward chords as follows.
Let $(v_i,v_j) \neq (v_k,v_l)$ be forward chords for $C$.
Then, $(v_i,v_j) \prec (v_k,v_l)$ if and only if $L_{ij} \subset L_{kl}$.
This relation inherits the property of being a partial order from the Boolean lattice.
We say that a forward chord $(v_i,v_j)$ is {\em minimal} if there is no forward chord $(v_k,v_l) \prec (v_i,v_j)$.
We say that the {\em first} minimal forward chord is a minimal forward chord $(v_i,v_j)$ such that $i$ is as small as possible.
Since there is at most one minimal forward chord starting at any vertex, the first minimal forward chord is unique.
See Figure \ref{fig:forwardChords}.

\begin{SCfigure}[][h]
\caption{The first minimal forward chord for this cycle is $v_1v_4$. The chord $v_7v_2$ is not forward, $v_1v_5$ is forward but not minimal, and $v_3v_6$ is forward and minimal but not first.}\label{fig:forwardChords}
\includegraphics[width=.6 \textwidth]{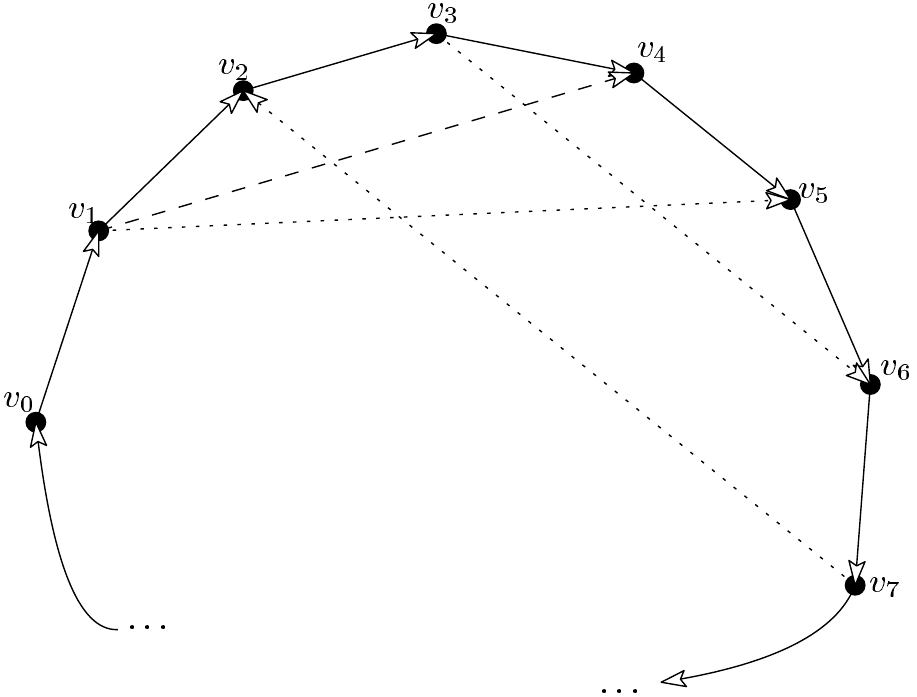}
\end{SCfigure}

We now describe how to map $C$ to a permutation.
If there are no forward chords with respect to $C$, we map the single Hamilton cycle of $G$ to the permutation consisting only of fixed points.
Otherwise, let $(v_s,v_t)$ be the first minimal forward chord.
We map $C$ to the permutation $P$ consisting of the cycle $C_{st}$ and the fixed points $\fix(P) = L_{st}$.

In order to show that the map we've described is injective, it will suffice to show that we can recover $C$ if we know $P$ and $v_0$.
To do this, we need to recover the identity of $v_s$ and $v_t$, and to find the order of $\fix(P)$ in $C$.

\begin{claim}
The first vertex in $C_{st}$ that has an edge into $\fix(P)$ is $v_s$, and $v_s$ has exactly one edge into $\fix(P)$.
\end{claim}
\begin{proof}
Suppose, for contradiction, that there is an edge $(v_i,v_j)$ with $i<s$ and $v_j \in \fix(P)$.
Then, either $(v_i,v_j)$ is minimal, or there is a minimal forward chord $(v_k, v_l) \prec (v_i,v_j)$.
If $(v_i,v_j)$ is minimal, then $(v_s,v_t)$ cannot be the first minimal forward chord, and we reach a contradiction.
If $(v_k, v_l) \prec (v_i,v_j)$, then either $k < s$ or $k \geq s$.
If $k < s$, then $(v_s, v_t)$ is not first, and we reach a contradiction.
If $k \geq s$, then $(v_s,v_t)$ is not minimal, and we reach a contradiction.

We still need to show that $v_s$ has only one edge into $\fix(P)$.
Suppose that there is an edge $(v_s,v_j) \in G$ with $j \neq s+1$ and $v_j \in \fix(P)$.
Then $(v_s,v_j)$ is a forward chord of $C$ with $(v_s,v_j) \prec (v_s,v_t)$, contradicting the choice of $(v_s,v_t)$ in the construction.
\end{proof}

Using this claim, we can identify $v_s$ and $v_t$, and we can identify $v_{s+1} \in \fix(P)$.
To do this, we simply start at $v_0$ and follow edges of $C_{st}$ until we find a vertex that has an edge into $\fix(P)$.
The vertex with an edge into $\fix(P)$ is $v_s$, the vertex in $\fix(P)$ that has an edge from $v_s$ is $v_{s+1}$, and the vertex that follows $v_s$ in $C_{st}$ is $v_t$.

Having identified $v_{s+1}$, we claim that we can determine the order of $\fix(P)$ in $C$.
Indeed, there are no forward chords in $\fix(P)$, since any such forward chord would precede $(v_s,v_t)$ in our partial order on forward chords, contradicting the minimality of $(v_s,v_t)$.
Hence, there is exactly one edge from any set $\{v_{s+1}, v_{s+2}, \ldots, v_{k}\}$ into $\{v_{k+1}, v_{k+2}, \ldots, v_{t-1}\}$.
Using this fact in an easy inductive argument, we can identify the order $\fix(P)$ in $C$.

It only remains to show that, under the assumption that $G$ is not a directed cycle, there is a choice of $v=v_0$ for which the identity permutation is not in the image of $f_v$.
It is already established that the permutation consisting of fixed points is only in the image of $f$ if there is exactly one Hamilton cycle $C$ on $G$.
If there is a chord $u,w$ of $C$, then $f_u(C)$ will include the cycle formed by $(u,w)$ together with edges of $C$, and hence the permutation consisting of fixed points will not be in the image of $f_u$. This completes the proof of Lemma \ref{th:HamiltonMap}.
\end{proof}

The proof of the main theorem is done.
Here is one other interesting consequence of Lemma \ref{th:HamiltonMap}.

\begin{corollary}\label{thm:hamilton}
For any directed graph $G=(E,V)$ that is not a directed cycle, there is a vertex $v \in V$ such that the number of cycles in $G$ that contain $v$ is at least twice the number of Hamilton cycles in $G$.
\end{corollary}

Indeed, for an appropriate choice of $v$, Lemma \ref{th:HamiltonMap} gives an explicit injection from the set of Hamilton cycles that contain $v$ to cycles that are not Hamilton.

If $G$ is a directed cycle together with an additional edge, then there are two cycles in $G$, one of which is Hamilton.
In this case, the conclusion of Corollary \ref{thm:hamilton} does not hold for those vertices not in the second cycle of $G$, and cannot be improved for those vertices that are in the second cycle of $G$.
The $2$-blowup of a directed cycle, discussed in Section \ref{sec:constructions}, gives a more interesting example for which the conclusion of Corollary \ref{thm:hamilton} is tight.
\section{Matchings in general graphs}\label{sec:matchingsInGeneralGraphs}
Let us now proceed to the proof of Theorem \ref{th:matchingInGeneral}.

\begin{proof}[Proof of Theorem \ref{th:matchingInGeneral}]
Let $G$ be a graph and $M=\{v_1u_1,v_2u_2,\ldots, v_nu_n\}$ be a perfect matching in $G$. By assigning $v_i,u_i$ one each to sets $L$ and $R$ for all $i$ we obtain $2^{n-1}$ different bipartitions of the vertex set. Let us denote by $G_1,\ldots, G_{2^{n-1}}$ the bipartite graphs obtained by only retaining the edges of $G$ between $L$ and $R$ for each of these bipartitions.

Note that $M$ is a perfect matching in each of the $G_i$ by construction. Let $a_i$ denote the number of perfect matchings of $G_i$ which intersect $M$ and $b_i$ the number of perfect matchings of $G_i$ that do not. By Theorem \ref{th:matchingInBipartiteGraph} we know $a_i \ge b_i$.

Furthermore, if $M'$ is any perfect matching of $G$ then $M \cup M'$ is bipartite, which means $M'$ is contained in at least one graph $G_i$. This implies that the number of perfect matchings in $G$ which don't intersect $M$ is upper bounded by $b_1+\cdots+b_n\le a_1+a_2+\cdots+a_n.$

On the other hand every perfect matching $M'$ of $G$ intersecting $M$ appears in at most $2^{n-1}$ of the $G_i$, so there are at least $(a_1+\cdots+a_n)/2^{n-1}$ such matchings. In particular, this is at least $1/2^{n-1}$ times the number of perfect matchings in $G$ which don't intersect $M$. This means that the proportion of perfect matchings which intersect $M$ is at least $1/(2^{n-1}+1)$ of the total number of perfect matchings as claimed.

\ignore{Let $G$ be any graph on $2n$ vertices and $M$ a perfect matching in $G$.  By selecting one vertex from each edge of $M$, there are $2^{n-1}$ ways to partition the vertices of $G$ into sets $\{L, R\}$ such that each edge of $M$ meets $L$ exactly once.  For each such bipartition, let $G[L,R]$ denote the bipartite subgraph of $G$ obtained by retaining only the edges that meet both $L$ and $R$.

Note that if $M'$ is any perfect matching of $G$, then $M \cup M'$ is bipartite and thus there exists a bipartition for which $M' \in G[L,R]$.  Using this fact followed by Theorem \ref{th:matchingInBipartiteGraph}, we obtain
\begin{eqnarray*}
|\{M' : M \cap M' = \emptyset\}| &\leq& \sum_{\{L,R\}} |\{M' :  M \cap M' = \emptyset \text{ and } M' \subseteq G[L,R]\}|\\
 &\leq& \sum_{\{L,R\}} |\{M' :  M \cap M' \neq \emptyset \text{ and } M' \subseteq G[L,R]\}|\\
  &\leq& 2^{n-1} |\{M' : M \cap M' \neq \emptyset\}|,
\end{eqnarray*}
which establishes the lower bound.}

\hrulefill

Our construction is the graph $H$ defined on $\{v_1,\ldots, v_{2n},$ $u_1,\ldots,u_{2n}\}$ as follows. For each $i=1,\ldots, n$ we put all the edges between $\{v_{2i-1},v_{2i}\}$ and $\{u_{2i-1},u_{2i}\}$ and make $v_{2i-1}v_{2i}$ and $u_{2i}u_{2i+1}$ edges (with $u_{2n+1}=u_1$).% Consider Figure \ref{fig:matchingExample} for an illustration.

\begin{figure}[h]
    \centering
    \begin{tikzpicture}[scale=0.6][line cap=round,line join=round]
	
	\defPt{-1}{-3}{u1}
	\defPt{1}{-3}{u2}
	\defPt{-1}{-5}{v1}
	\defPt{1}{-5}{v2}
	
	\defPt{3}{-1}{u3}
	\defPt{3}{1}{u4}
	\defPt{5}{-1}{v3}
	\defPt{5}{1}{v4}

	\defPt{1}{3}{u5}
	\defPt{-1}{3}{u6}
	\defPt{1}{5}{v5}
	\defPt{-1}{5}{v6}

	\defPt{-3}{1}{u7}
	\defPt{-3}{-1}{u8}
	\defPt{-5}{1}{v7}
	\defPt{-5}{-1}{v8}

	\def \shrinkx {3}
	\def \r {0.8}
	\pgfmathsetmacro{\p}{\r*\shrinkx}

	\draw[line width=2pt] (u1) -- (v1) -- (u2) -- (v2) -- (u1);
	\draw[line width=2pt] (u3) -- (v3) -- (u4) -- (v4) -- (u3);
	\draw[line width=2pt] (u5) -- (v5) -- (u6) -- (v6) -- (u5);
	\draw[line width=2pt] (u7) -- (v7) -- (u8) -- (v8) -- (u7);
	
	\draw[line width=2pt,red] (v1) -- (v2);
	\draw[line width=2pt,red] (v3) -- (v4);
	\draw[line width=2pt,red] (v5) -- (v6);
	\draw[line width=2pt,red] (v7) -- (v8);
	
	\draw[line width=2pt,red] (u2) -- (u3);
	\draw[line width=2pt,red] (u4) -- (u5);
	\draw[line width=2pt,red] (u6) -- (u7);
	\draw[line width=2pt,red] (u8) -- (u1);

	\foreach \i in {1,...,8}
	{
		\draw (u\i) \smvx;
		\node[] at ($0.85*(u\i)$) {$u_{\i}$};
		\draw (v\i) \smvx;
	    \node[] at ($1.1*(v\i)$) {$v_{\i}$};
	}

	\end{tikzpicture}
    \caption{The construction in Theorem \ref{th:matchingInGeneral} for $n=4$. The red edges form $M_0$.}
    \label{fig:matchingExample}
\end{figure}
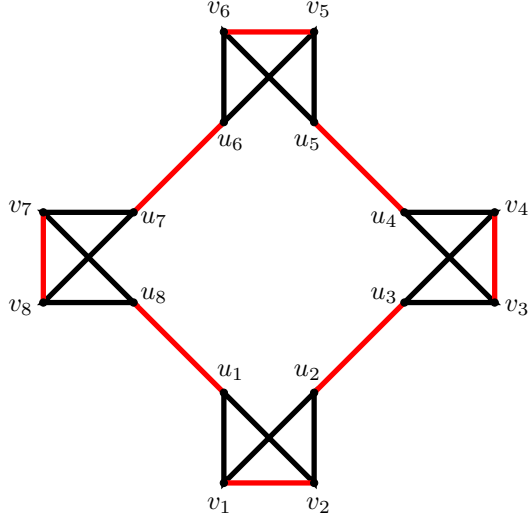

Let $M_0$ be the perfect matching in $G$ given by $\{v_1v_2,\ldots, v_{2n-1}v_{2n}\}$ and $\{u_2u_3,u_4u_5,\ldots, u_{2n}u_1\}$. Notice that if for each $i$ we choose either $v_{2i-1}u_{2i-1}$ and $v_{2i}u_{2i}$ or $v_{2i-1}u_{2i}$ and $v_{2i}u_{2i-1}$ we get a perfect matching in $G$, so there are $2^n$ perfect matchings disjoint from $M_0$.  And in fact $H$ has no other perfect matchings, since (perhaps after some thought) it is clear that any perfect matching containing an edge of $M_0$ must coincide with $M_0$.
%We claim that $H$ has no matchings that intersect $M_0$ nontrivially.  Suppose for instance that $v_{2i-1}v_{2i}$ is in some perfect matching $M'$ then $u_{2i}$ must be matched to $u_{2i+1}$ since its only other neighbours are $v_{2i-1}$ and $v_{2i}$ which are already matched. So $u_{2i}u_{2i+1}$ is also in $M'$. On the other hand if $u_{2i}u_{2i+1}$ is in $M'$ then $v_{2i+1}$ must be matched to either $u_{2i+2}$ or $v_{2i+2}$. In the former case $v_{2i+2}$ can not be matched to any other vertex so the latter case must always hold. By repeatedly applying these two observations we deduce that if $M'$ intersects $M$ it must be equal to $M$.
\end{proof}

The argument for our lower bound could perhaps be improved in two ways.  First, note that every perfect matching $M'$ appears in $2^{c-1}$ bipartitions $G[L,R]$, where $c$ is the number of components of $M \cup M'$.  For matchings disjoint with $M$, we used the crude lower bound of $1$ and for the others an upper bound of $2^{n-1}$.  Although this may feel like giving up quite a lot, these were actually tight in our construction $H$.  The second possible area for improvement would be to remove the slack incurred in our application of Theorem \ref{th:matchingInBipartiteGraph}.

\section{Random (di)graphs}\label{sec:randomGraphs}
Here, we outline the proof of Proposition \ref{th:randomRatios}.  Our approach for this is a routine second-moment method in the spirit of Janson's \cite{janson1994numbers}, where the main idea is to first condition on the number of edges (resp.\ arcs).  We direct the readers to \cite{alon2004probabilistic} for a general text on probabilistic combinatorics and to the papers \cite{janson1994numbers} and \cite{frieze1995analysis} for details of the argument we outline below.

\begin{proof}
We first outline the proof for the directed graph case, discussing modifications needed for the graph case at the end.  Following the ideas in \cite{janson1994numbers, frieze1995analysis}, let $DG = DG_{n,m}$ denote a uniformly random digraph on $n$ vertices having exactly $m$ arcs.  Let $X$ denote the number of derangements of $DG$ and let $Y$ denote the number of permutations.  Note for a digraph $H$ on $[n]$ having $t$ arcs, the probability that $H$ is contained in $DG$ depends only on $t$ and is given by
\[
\mathbb{P}(H \subseteq DG) =f(t) := {n(n-1) - t \choose m-t} {n(n-1) \choose m}^{-1}.
\]
For $1 \leq t \leq 2n$ (and $\varepsilon n^2 /2 \leq m \leq n^2$), this is given by
\[
\mathbb{P}(H \subseteq DG) =f(t) = \left(\dfrac{m}{n(n-1)} \right) ^{t} \exp \left\{ \dfrac{-t^2}{2} \left(\dfrac{1}{m} - \dfrac{1}{n(n-1)} \right) + \mathcal{O}(1/n) \right\}.
\]
When $t=n$, this gives the probability that a given derangement is contained in $DG$.  Let $Der(n)$ denote the number of derangements of the set $[n]$, and define $\gamma = m / [n(n-1)]$.  Using that $n! /e-1 \leq Der(n) \leq n! /e+1$, we have
\begin{eqnarray*}
\mathbb{E}[X] &=& f(n) Der(n) = \gamma^n \exp \left[ \left(\dfrac{-n^2}{2m} + \dfrac{n}{2(n-1)} \right) + \mathcal{O}(1/n) \right] Der(n)\\
&=& \gamma ^n \exp \left[ \dfrac{-1}{2\gamma} + \dfrac{1}{2} + \mathcal{O}(1/n) \right]Der(n) = n! \gamma ^n \exp \left[\dfrac{-1}{2\gamma} - \dfrac{1}{2} + \mathcal{O}(1/n) \right]
\end{eqnarray*}
as the expected number of derangements of $DG_{n,m}$.

Similarly, if $\sigma$ is a permutation of $[n]$ with $k$ fixed points, then the probability that $\sigma$ is a permutation of $DG$ is $f(n-k)$.  Therefore, the expected number of permutations of $DG$ is given by
\begin{eqnarray*}
\mathbb{E}[Y] &=& \sum_{k = 0} ^{n} f(n-k) {n \choose k} Der(n-k)\\
&=& n! \gamma^n \exp\left[\dfrac{-1}{2\gamma} - \dfrac{1}{2} + \mathcal{O}(1/n)\right]\\
& & \qquad \qquad \times \sum_{k=0} ^{n} \dfrac{(1/\gamma)^{k}}{k!} \exp\left[\dfrac{2kn - k^2}{2} \left(\dfrac{1}{m} - \dfrac{1}{n(n-1)} \right) \right]\\
&=& n! \gamma^n \exp\left[\dfrac{1}{2\gamma} - \dfrac{1}{2} + \mathcal{O}(\log(n)/n)\right],
\end{eqnarray*}
where in the last line, we estimate the sum by the Taylor series for $e^{1/\gamma}$.  The slightly enlarged (suboptimal) error term comes from a very crude bound estimating the sum based on whether or not $(1/\gamma)^k / k! \leq 1/n$.

After this, to finish the directed graph case we need only estimate the variances of $X$ and $Y$.  Routine computations almost identical to those in \cite{frieze1995analysis} show that $\mathbb{E}[X^2] = \mathbb{E}[X]^2 (1 + o(1))$ and $\mathbb{E}[Y^2] = \mathbb{E}[Y]^2 (1 + o(1))$.  Therefore, by Chebyshev's inequality, with high probability $X/Y = (1 \pm o(1)) \mathbb{E}[X] / \mathbb{E}[Y] = (1 \pm o(1)) e^{-1/\gamma}.$  The proof concludes by the standard coupling of $DG_{n,q}$ and $DG_{n,m}$ using the fact that $\gamma = q + o(1)$ with high probability (since the number of arcs of $DG_{n,q}$ is well concentrated).  See \cite{janson1994numbers, frieze1995analysis} for full details of this method.

The graph case is almost identical but with $\gamma = 2m / n^2$.  In the directed graph case, a guiding heuristic in computing moments is that two random derangements should share approximately as many edges as two random permutations.  Whereas in the graph case, the heuristic is that for a random permutation, the number of fixed points and the number of cycles of length 2 should be approximately independent.
\end{proof}

\section{Open problems}\label{sec:openProblems}

Finally, we mention a few open problems, mostly inspired by studying derangements and permutations on graphs.

\begin{enumerate}
	\item Theorem \ref{th:matchingInGeneral} shows that in general, a perfect matching may intersect with only an exponentially small proportion of all perfect matchings, but there is a sizable gap between our construction and the general lower bound.  What is the correct base for the exponent?
	\item Let $S = \{(d/p)_G \ : \ \text{$G$ is a digraph}\}$ be the set of values arising as a ratio $(d/p)$.  We have shown that $S \subseteq [0,1/2]$ and that $S$ is dense in $[0,1/e]$ (by taking the appropriate Erd\H{o}s--Renyi graphs).  Is $S$ dense in $[0,1/2]$?  If so, is $S$ equal to $[0,1/2] \cap \mathbb{Q}$?
	\item What can be said about $(d/p)$ for regular directed graphs with degree greater than $n/2$, but much less than $n$?
\end{enumerate}

\bibliographystyle{plain}
\bibliography{derangements}{}

\end{document}